\documentclass{amsart}[12pt]
\usepackage{amsmath, amsthm, amscd, amsfonts,}
\usepackage{graphicx,float}

\usepackage{amssymb}
\usepackage[utf8]{inputenc}

\DeclareMathOperator{\Flim}{Flim}

\def\MPB{{\mathbb{P}}}

\def\MRB{{\mathbb{R}}}

\def\k{\kappa}

\def\a{\alpha}

\def\b{\beta}

\newtheorem{theorem}{Theorem}[section]
\newtheorem{lemma}[theorem]{Lemma}

\newtheorem{definition}[theorem]{Definition}

\newtheorem{remark}[theorem]{Remark}

\numberwithin{equation}{section}

\usepackage{pb-diagram}

\def\rmark{\mbox{$\rm\bf\rule{0.06em}{1.45ex}\kern-0.05em R$}}
\def\pmark{\mbox{$\rm\bf\rule{0.06em}{1.45ex}\kern-0.05em P$}}
\def\nmark{\mbox{$\rm\bf\rule{0.06em}{1.45ex}\kern-0.05em N$}}
\def\vdash{\mbox{$\rm\| \kern-0.13em -$}}

\def\rmark{\mbox{$\rm\bf\rule{0.06em}{1.45ex}\kern-0.05em R$}}
\def\pmark{\mbox{$\rm\bf\rule{0.06em}{1.45ex}\kern-0.05em P$}}
\def\nmark{\mbox{$\rm\bf\rule{0.06em}{1.45ex}\kern-0.05em N$}}
\def\vdash{\mbox{$\rm\| \kern-0.13em -$}}
\newcommand{\lusim}[1]{\smash{\underset{\raisebox{1.2pt}[0cm][0cm]{$\sim$}}
{{#1}}}}

\begin{document}

\title[Fra\"{i}ss\'{e} limit via forcing]{Fra\"{i}ss\'{e} limit via forcing}

\author[M. Golshani ]{Mohammad
  Golshani }


~
~~
~
\thanks{The  author's research has been supported by a grant from IPM (No. 97030417).} \maketitle

\begin{abstract}
Given a Fra\"{i}ss\'{e} class $\mathcal{K}$ and an infinite cardinal $\kappa,$ we define a forcing notion
which adds a structure of size $\kappa$ using elements of $\mathcal{K}$, which extends the Fra\"{i}ss\'{e} construction in the case $\k=\omega.$
\end{abstract}

\section{introduction}
Suppose $\mathcal{L}$ is a finite relational language and $\mathcal{K}$ is a class of finite $\mathcal{L}$-structures closed under substructures and isomorphisms. It is called a
Fra\"{i}ss\'{e} class if it satisfies Joint Embedding Property (JEP)  and Amalgamation Property (AP). A Fra\"{i}ss\'{e} limit, denoted $\Flim(\mathcal{K})$, of a
Fra\"{i}ss\'{e} class $\mathcal{K}$ is the unique\footnote{The existence and uniqueness follows from Fra\"{i}ss\'{e}'s theorem, See \cite{hodges}.} countable ultrahomogeneous (every isomorphism of finitely-generated substructures extends to an automorphism of $\Flim(\mathcal{K})$) structure into which every member of $\mathcal{K}$ embeds.

Given  a Fra\"{i}ss\'{e} class $\mathcal{K}$  and an infinite cardinal $\kappa,$ we would like to force a structure of size $\kappa$
which shares many properties with the Fra\"{i}ss\'{e} construction.

We first consider the special cases of linear orders, and then discuss the more general case. We also show that our construction
gives the original
Fra\"{i}ss\'{e} construction in the case $\k=\omega.$
\begin{remark}
The results of this paper can  be proved for Hrushovski's construction as well. We leave the details to the interested reader.
\end{remark}
\section{Fra\"{i}ss\'{e} limit of linear orders}
Let $\mathcal{K}$ be the class of finite linear orders and let $\kappa$ be an infinite cardinal.
\begin{definition}
A condition in $\MPB_{\k, \mathcal{K}}$ is of the form $p=(A_p, \leq_p)$, where
\begin{enumerate}
\item $A_p$ is a finite subset of $\kappa.$
\item $(A_p, \leq_p) \in \mathcal{K}$, i.e., it is a finite linear order.
\end{enumerate}
\end{definition}
The order on  $\MPB_{\k, \mathcal{K}}$  is defined in the natural way.
\begin{definition}
Suppose $p, q \in \MPB_{\k, \mathcal{K}}$. Then $p \leq q$ if and only if
\begin{enumerate}
\item $A_p \supseteq A_q,$
\item $\leq_q = \leq_p \cap (A_q \times A_q)$.
\end{enumerate}
\end{definition}
\begin{lemma}
\label{chain condition for linear orders}
$\MPB_{\k, \mathcal{K}}$   is c.c.c. (and in fact $\aleph_1$-Knaster).
\end{lemma}
\begin{proof}
Suppose $\{p_\a=(A_\a, <_\a): \a < \omega_1   \} \subseteq \MPB_{\k, \mathcal{K}}$. By $\Delta$-system lemma, we may assume that $(A_\a: \a< \omega_1)$
forms a $\Delta$-system with root $d$. Since there are only finitely many different orders on $d$, so we can find $I \subseteq \omega_1$ of size $\aleph_1$
such that for all $\a < \b$ in $I$, we have $<_a \cap ~d \times d = <_\b \cap ~d \times d.$ But then $q_{\a, \b}=(A_\a \cup A_\b, <_\a \cup <_\b)$ is a condition extending both
of $p_\a$ and $p_\beta.$
\end{proof}
It follows that forcing with $\MPB_{\k, \mathcal{K}}$  preserves cardinals and cofinalities. Let $G$ be $\MPB_{\k, \mathcal{K}}$-generic over $V$.
The next lemma follows by a simple density argument.
\begin{lemma}
$\k=\bigcup_{p \in G}A_p.$
\end{lemma}
Let
$\leq_G = \bigcup_{p \in G} \leq_p.$
As $G$ is a filter, $\leq_G$ is well-defined and it is a linear order on $\k.$
\begin{lemma}
\label{kappa-density}
$(\k, \leq_G)$ is $\k$-dense.
\end{lemma}
\begin{proof}
Suppose $\alpha <_G \beta$ and $p \in  \MPB_{\k, \mathcal{K}}$. Suppose $p \Vdash$``$\{ \gamma \in \k: \a <_G \gamma <_G \b      \}$ has size $< \k$''.
By Lemma \ref{chain condition for linear orders},
we can find $I \in V$ of size less than $\k$ such that $p \Vdash$``$\check I \supseteq \{ \gamma \in \k: \a <_G \gamma <_G \b      \}$''. Let $q \leq p$ be such that
$\a, \b \in A_q$ and for some $\gamma \in A_q \setminus I, \a <_q \gamma <_q \b.$
Then
$q \Vdash$``$\gamma \notin I$ and $\a \lusim <_G \gamma \lusim <_G \b  $'',
which is a contradiction.
\end{proof}
\begin{lemma}
\label{having countable dense subsets}
Suppose $I \in V$ is an infinite subset of $\k.$ Then $(I, <_G)$ is dense in $(\k, <_G)$.
\end{lemma}
\begin{proof}
Suppose $p \in \MPB_{\k, \mathcal{K}}$ and $p \Vdash$``$\a \lusim <_G \b$''. Let $\gamma \in I \setminus A_p$ and let $q=(A_q, \leq_q),$
where $A_q=A_p \cup \{\gamma\}$ and $\a <_q \gamma <_q \b.$ Then
$q \Vdash$``$\gamma \in I$ and $\a \lusim <_G \gamma \lusim <_G \b$''.
\end{proof}
\begin{remark}
Work in $V[G].$ Let $\overline{(\k, <_G)}$ denote the completion of $(\k, <_G).$ Then $\overline{(\k, <_G)} \cong (\MRB^{V[G]}, <)$, where $(\MRB^{V[G]}, <)$ denotes the set of real numbers in $V[G]$.
\end{remark}
\begin{lemma}
Suppose $I, J \in V$ are non-empty subsets of $\k$ and $I <_G J$ (i.e., $\a <_G \b$ for all $\a \in I$ and $\b \in J$). Then $I$ and $J$ are finite, in particular, there exists $\gamma \in \k$ such that $I <_G \gamma <_G J.$
\end{lemma}
As the next lemma shows, the structure  $(\k, <_G)$ is rigid if $\kappa$ is uncountable.
\begin{lemma} (\cite[Theorem 2]{kostana})
\label{finite homogeneity}
Suppose $\kappa$ is an uncountable cardinal. Then $(\kappa, <_G)$ is rigid.
\end{lemma}


\section{Fra\"{i}ss\'{e} limit--The general case}
Suppose $\mathcal{L}$ is a finite relational language and $\mathcal{K}$ is a Fra\"{i}ss\'{e} class. For any
 relation symbol $R \in \mathcal{L}$, let $n_R$ denote its arity.

\begin{definition}
A condition $p$ is in $\MPB_{\k, \mathcal{K}}$ if and only if
\begin{enumerate}
\item $p \in \mathcal{K}$.
\item $A_p$, the universe of the structure $p$, is a  subset of $\k.$
\end{enumerate}
\end{definition}
The order on  $\MPB_{\k, \mathcal{K}}$  is defined in the natural way.
\begin{definition}
Suppose $p, q \in \MPB_{\k, \mathcal{K}}$. Then $p \leq q$ if and only if
\begin{enumerate}
\item $A_p \supseteq A_q,$
\item $q = p \restriction A_q$, .i.e., for any relational symbol $R \in \mathcal{L}, R^q = R^p \cap A_q^{n_R}$.
\end{enumerate}
\end{definition}
\begin{lemma}
\label{chain condition}
$\MPB_{\k, \mathcal{K}}$   is c.c.c. (and in fact $\aleph_1$-Knaster).
\end{lemma}
\begin{proof}
Suppose $\{p_\a=(A_\a, <_\a): \a < \omega_1   \} \subseteq \MPB_{\k, \mathcal{K}}$. By $\Delta$-system lemma, we may assume that $(A_\a: \a< \omega_1)$
forms a $\Delta$-system with root $d$. Since there are only countably many $\mathcal{K}$-structures with universe $d$, so we can find $I \subseteq \omega_1$ of size $\aleph_1$
such that for all $\a < \b$ in $I$, we have $p \restriction d  =q \restriction d.$ But then, using the amalgamation property, we can find $q_{\a, \b} \in \mathcal{K}$
which extends both of $p_\a$ and $q_\a.$
\end{proof}
It follows that forcing with $\MPB_{\k, \mathcal{K}}$  preserves cardinals and cofinalities. Let $G$ be $\MPB_{\k, \mathcal{K}}$-generic over $V$.
The next lemma follows by a simple density argument.
\begin{lemma}
$\k=\bigcup_{p \in G}A_p.$
\end{lemma}
For any relational symbol $R \in \mathcal{L}$ let
$R^G = \bigcup_{p \in G} R^p,$
where $R^p$ is the interpretation of $R$ in $p$.
As $G$ is a filter, $R^G$ is a well-defined $n_R$-ary relation on $\k.$
Consider the structure
\[
\mathcal{M}_G= (\k, R^G)_{R \in \mathcal{L}}.
\]
Then $\mathcal{M}_G$ is an $\mathcal{L}$-structure with universe $\k.$
\begin{lemma}
\label{embeddability}
Each element of $\mathcal{K}$ embeds into $\mathcal{M}_G$.
\end{lemma}
\begin{proof}
Suppose $p \in \mathcal{K}$. We may assume that $A_p$, the universe of the structure $p$, is a  subset of $\k.$ Then the set
\[
D=\{q \in \mathcal{K}: p \text{~embeds into~}q       \}
\]
is dense, from which the lemma follows.
\end{proof}
The next lemma can be proved by a simple density argument.
\begin{lemma}
\label{restriction to ground model sets}
Suppose $I \in V$ is an infinite subset of $\k.$ Then for any $R \in \mathcal{L}$ and any $(\a_1, \dots, \a_i) \in \k^{i}, i < n_R$ the sets
\[
\{ (a_{i+1}, \dots, a_{n_R}) \in I^{n_R - i}: R^G(a_1, \dots, a_{n_R})                                \}
\]
and
\[
\{ (a_{i+1}, \dots, a_{n_R}) \in I^{n_R - i}: \neg R^G(a_1, \dots, a_{n_R})                                \}
\]
are infinite; and in fact they have the same size as $I$.
\end{lemma}

\section{Connection with the original  Fra\"{i}ss\'{e} construction}
In this section we show that the Fra\"{i}ss\'{e} construction can be obtained from the above results for the case $\kappa=\omega.$

Thus fix the language $\mathcal{L}$ and the class $\mathcal{K}$ as before. We may suppose that each $p \in \mathcal{K}$ has domain a subset of $\omega$,
so in particular $\mathcal{K}$ is countable. Let $\MPB=\MPB_{\omega, \mathcal{K}}$. We define the following dense subsets of $\MPB$:
\begin{enumerate}
\item $D_n=\{ p \in \MPB: n \in A_p           \}$, for $n<\omega$.
\item $D^{\rightarrow}_{(\mathcal{A}, \mathcal{A}'), \mathcal{B}, f}=\{p \in \MPB: p \restriction A =\mathcal{A}, A' \subseteq  A_p$ and $\exists~ q$ such that $ q \restriction B =\mathcal{B}$ and
 $f$ extends to some $\bar f: p \cong      q            \}$, where $\mathcal{A}=(A, \dots), \mathcal{B}=(B, \dots), \mathcal{A}'=(A', \dots) \in \mathcal{K},
 \mathcal{A}$ is a substructure of $\mathcal{A}'$ and $f: \mathcal{A} \cong \mathcal{B}$.

\item $D^{\leftarrow}_{\mathcal{A}, (\mathcal{B}, \mathcal{B}'), f}=\{q \in \MPB: q \restriction B =\mathcal{B}, B' \subseteq A_q$ and $\exists ~ p$ such that $ p \restriction A =\mathcal{A}$ and
 $f$ extends to some $\bar f: p \cong      q            \}$, where $\mathcal{A}=(A, \dots), \mathcal{B}=(B, \dots), \mathcal{B}'=(B', \dots) \in \mathcal{K}, \mathcal{B}$ is   a substructure of $\mathcal{B}'$ and  $f: \mathcal{A} \cong \mathcal{B}$.
\end{enumerate}
Note that
\[
\{ D_n\} \cup \{ D^{\rightarrow}_{(\mathcal{A}, \mathcal{A}'), \mathcal{B}, f}\} \cup \{D^{\leftarrow}_{\mathcal{A}, (\mathcal{B}, \mathcal{B}'), f}\}
\]
is countable and hence
 by the Rasiowa-Sikorski lemma,  we can find a filter $G \subseteq \MPB$ which meets all the above dense sets.
Then the resulting structure $\mathcal{M}_G$ is isomorphic to the Fra\"{i}ss\'{e} limit of the class $\mathcal{K}$.

Mohammad Golshani,
School of Mathematics, Institute for Research in Fundamental Sciences (IPM), P.O. Box:
19395-5746, Tehran-Iran.

E-mail address: golshani.m@gmail.com

URL: http://math.ipm.ac.ir/golshani/

\end{document}